\documentclass[10pt]{article}
\usepackage[T1]{fontenc}
\usepackage[cp1250]{inputenc}
\usepackage{lmodern}

\usepackage[english]{babel}
\usepackage{latexsym}
\usepackage{amsmath}
\usepackage{indentfirst}
\usepackage{amsthm}
\usepackage{amssymb}
\usepackage{amsfonts}
\usepackage{stackrel}
\usepackage{dsfont}
\usepackage{tikz}
\usepackage{slashbox}
\usepackage[mathscr]{eucal}
\usepackage{authblk}
\usepackage{hyperref}
\usepackage{mathabx}
\usepackage{bookmark}

\newtheorem{thm}{Theorem}
\newtheorem{lemma}[thm]{Lemma}
\newtheorem{cor}[thm]{Corollary}

\newcommand{\reals}{\mathbb{R}}
\newcommand{\naturals}{\mathbb{N}}
\newcommand{\integers}{\mathbb{Z}}
\newcommand{\complex}{\mathbb{C}}

\newcommand{\eps}{\varepsilon}
\newcommand{\supp}{\text{supp}}
\newcommand{\hh}{\mathfrak{H}}

\newcommand{\Ffamily}{\mathcal{F}}
\newcommand{\algebra}{\mathcal{A}}
\newcommand{\BB}{\mathfrak{B}}
\newcommand{\Hilbert}{\mathcal{H}}

\begin{document}
\title{How Arzel\`a and Ascoli would have proved Pego theorem for $L^1(G)$ (if they lived in the $21^{\text{st}}$ century)?}
\author{Mateusz Krukowski}
\affil{Institute of Mathematics, \L\'od\'z University of Technology, \\ W\'ol\-cza\'n\-ska 215, \
90-924 \ \L\'od\'z, \ Poland \\ \vspace{0.3cm} e-mail: mateusz.krukowski@p.lodz.pl}
\maketitle

\begin{abstract}
In the paper we make an effort to answer the question ``What if Arzel\`a and Ascoli lived long enough to see Pego theorem?''. Giulio Ascoli and Cesare Arzel\`a died in 1896 and 1912, respectively, so they could not appreciate the characterization of compact families in $L^2(\reals^N)$ provided by Robert L. Pego in 1985. Unlike the Italian mathematicians, Pego employed various tools from harmonic analysis in his work (for instance the Fourier transform or the Hausdorff-Young inequality). Our article is meant to serve as a bridge between Arzel\`a-Ascoli theorem and Pego theorem (for $L^1(G)$ rather than $L^2(G)$, $G$ being a locally compact abelian group). In a sense, the former is the ``raison d'\^etre'' of the latter, as we shall painstakingly demonstrate.
\end{abstract}

\smallskip
\noindent 
\textbf{Keywords : } Banach algebras, $C^*-$algebra of a locally compact abelian group, Arzel\`a-Ascoli theorem, Pego theorem, equicontinuity and equivanishing, compactness in function spaces\\
\vspace{0.2cm}
\\
\textbf{AMS Mathematics Subject Classification (2010): } 43A15, 43A20, 43A25, 46J10, 46L05

\section{Introduction}

In 1985 Robert L. Pego came up with the following characterization of compact families in $L^2(\reals^N):$

\begin{thm}(classical Pego theorem, comp. \cite{Pego})\\
A bounded family $\Ffamily\subset L^2(\reals^N)$ is relatively compact if and only if 
\begin{itemize}
	\item $\widehat{\Ffamily}$ is equicontinuous, i.e. for every $\eps>0$ there exists an open neighbourhood $U_0\in\reals^N$ of the zero vector such that
	$$\forall_{\substack{y\in U_0\\ f\in\Ffamily}}\ \int_{\reals^N}\ |\widehat{f}(x+y) - \widehat{f}(x)|^2 < \eps,$$
	\item $\widehat{\Ffamily}$ is equivanishing, i.e. for every $\eps>0$ there exists a compact set $K$ (we denote it by $K\Subset \reals^N$) such that 
	$$\forall_{f\in\Ffamily}\ \int_{\reals^N\backslash K}\ |\widehat{f}(x)|^2 < \eps.$$
\end{itemize}
\label{classicalPego}
\end{thm}

The heart of Pego's reasoning is that equicontinuity in the ``time domain'' is equivalent to equivanishing in the ``frequency domain'' (i.e. the domain of the Fourier transform). The same is true if we look at equivanishing in the ``time domain'' and equicontinuity in the ``frequency domain'' $-$ they are also equivalent. If we couple these observations with the Kolmogorov-Tamarkin-Riesz theorem (comp. \cite{HancheOlsenHolden}), we obtain an elegant characterization of compact families in $L^2(\reals^N)$ via the Fourier transform. 

Nearly 30 years after Pego's discovery, Przemys\l{}aw G\'orka published a paper (comp. \cite{Gorka}) in which he generalized Theorem \ref{classicalPego} by replacing $\reals^N$ with a locally compact abelian group $G$. Originally, G\'orka's result contained a rather technical restriction, which the author (collaborating with Tomasz Kostrzewa) proved to be obsolete 2 year later (comp. \cite{GorkaKostrzewa}) . G\'orka's approach bears a striking resemblance to Pego's $-$ it employs Plancherel theorem to demonstrate the ``equicontinuity$-$equivanishing swap'' and invokes Weil's theorem (comp. Chapter 12 in \cite{Weil}, p. 52), which generalizes Kolmogorov-Tamarkin-Riesz theorem. 

All three mathematicians (Pego, G\'orka and Kostrzewa) view the Fourier transform as a linear and bounded operator on ``$L^2-$space'' (be it $L^2(\reals^N)$ or $L^2(G)$). However, the Fourier transform is originally defined on ``$L^1-$space'' and only then it is extended to ``$L^2-$space''. This motivates the question: ``\textit{Is there a counterpart of the Pego theorem for $L^1(G)?$}'' Our paper is devoted to answering this question in the affirmative. 

In Chapter \ref{Chapter:Banachalgebras} we briefly recall the notions related to Banach ($*-$)algebras and $C^*-$algebras. The chapter contains a reflextion on the question ``\textit{What if $L^1(G)$ were a $C^*-$algebra?}'' We argue that if it were the case, we could apply the celebrated Gelfand-Naimark theorem directly to $L^1(G),$ which would save us a considerable amount of effort. Unfortunately, $L^1(G)$ is not a $C^*-$algebra and overcoming this hurdle is the primary goal of the consecutive chapter. 

Chapter \ref{Chapter:Cstaralgebra} commences with the introduction of the convolution operator $\Psi.$ Our main source of reference at this point is Deitmar's and Echterhoff's ``Principles of Harmonic analysis'', although we also suggest an alternative approach to the construction of $\Psi.$ The purpose of the convolution operator is to ``copy'' $L^1(G)$ into the $C^*-$algebra $\BB(L^2(G))$ as ``faithfully'' as possible. Theorem \ref{normofconvolutionoperator} and Corollary \ref{Psinotisometry} that it may happen that $\|\Psi(F)\| < \|f\|_1,$ so $C^*(G)$ (the ``copy'' of $L^1(G)$) is not as faithful as we would wish it to be. Nevertheless, $\Psi$ is still an injective algebra $*-$homomorphism $-$ this is not ideal, but definitely very far from tragic. 

Chapter \ref{Chapter:ArzelaAscoliPego} is the final part of our investigations. The beginning of the chapter focuses on Arzel\`a-Ascoli theorem, whose classic version is well-established in the literature. Unfortunately, as far as the author is aware, the same remark does not apply for the generalizations of this result. In the paper, we characterize relatively compact families in $C_0(X)$ (where $X$ is a locally compact group) via the ``collapse-of-the-topologies'' technique. This method can be traced as far back as the classic monographs by Kelley and Munkres (comp. \cite{Kelley} and \cite{Munkres}). Once we have the Arzel\`a-Ascoli theorem at our disposal, we apply this result to characterize relatively compact families in $L^1(G)$ via the Fourier transform. This counterpart of the classic Pego theorem is the climax of the paper.

\section{Beyond Banach algebras}
\label{Chapter:Banachalgebras}

As the ancient proverb goes, ``\textit{A journey of a thousand miles begins with a single step}'' (comp. \cite{LaoTse}, \mbox{Chapter 64}). In accordance with the Chinese wisdom, our journey has to begin somewhere and we feel that the realm of Banach algebras is a good place to start. Following Eberhard Kaniuth's ``A Course in Commutative Banach Algebras'' (comp. \cite{Kaniuth}, p. 1) we say that a normed linear space $(\algebra,\|\cdot\|)$ over the field of complex numbers $\complex$ is a \textit{normed algebra} if it is an algebra and the norm is \textit{submultiplicative}, i.e.
$$\forall_{f,g \in \algebra}\ \|f\cdot g\| \leqslant \|f\|\cdot \|g\|.$$

\noindent
A normed algebra $\algebra$ is said to be a \textit{Banach algebra} if it is a Banach space (a plethora of examples of Banach algebras is provided in \cite{Bobrowski}, Chapter 6). A map $\Psi:\algebra_1\longrightarrow \algebra_2$ between two Banach algebras $\algebra_1$ and $\algebra_2$ is a \textit{Banach algebra homomorphism} if it is continuous, $\complex-$linear and \textit{multiplicative}, i.e.
$$\forall_{f,g\in\algebra_1}\ \Psi(f\cdot g) = \Psi(f)\cdot \Psi(g).$$

With every commutative Banach algebra $\algebra$, one can associate the so-called \textit{structure space} $\Delta_{\algebra}$ (comp. \cite{DeitmarEchterhoff}, p. 43 or \cite{Folland}, p. 5 or \cite{Kaniuth}, p. 46). It is the set of all non-zero Banach algebra homomorphisms $m:\algebra \longrightarrow \complex.$ The elements of $\Delta_{\algebra}$ are typically referred to as the \textit{multiplicative linear functionals}, which is why we use the letter $m$ to denote them.




Suppose that $\algebra$ is a Banach algebra and that $*:\algebra\longrightarrow \algebra$ is a map such that for every $f,g\in\algebra,\ \lambda\in\complex$ we have 
\begin{itemize}
	\item $(f+g)^* = f^* + g^*,$
	\item $(\lambda f)^* = \overline{\lambda}f^*,$
	\item $(f\cdot g)^* = g^*\cdot f^*,$
	\item $(f^*)^* = f,$
	\item $\|f^*\| = \|f\|.$
\end{itemize}

\noindent
The pair $(\algebra,*)$ is called a \textit{Banach $*-$algebra} and the map $*:\algebra\longrightarrow\algebra$ is called an \textit{involution} (comp. \cite{DeitmarEchterhoff}, p. 49). A map $\Psi:\algebra_1\longrightarrow \algebra_2$ between two Banach $*-$algebras $\algebra_1$ and $\algebra_2$ is a \textit{Banach $*-$algebra homomorphism} if, in addition to being a Banach algebra homomorphism, it satisfies
$$\forall_{f\in\algebra_1}\ \Psi\left(f^*\right) = \Psi(f)^*.$$

Arguably the most common example of a Banach $*-$algebra is $\BB(\Hilbert),$ the space of bounded and linear operators on a given Hilbert space $\Hilbert$. The involution $*:\BB(\Hilbert)\longrightarrow\BB(\Hilbert)$ maps every operator $T$ to its \textit{adjoint operator} (comp. Theorem 4.10 in \cite{Rudin} or Proposition 5.4 in \cite{SteinShakarchi}, p. 183), i.e. a unique bounded and linear operator  $T^*$ satisfying 
$$\forall_{u,v\in\Hilbert}\ \langle Tu|v \rangle = \langle u|T^*v \rangle,$$

\noindent
where $\langle\cdot|\cdot\rangle$ denotes the inner product in $\Hilbert.$ The relevance of $\BB(\Hilbert)$ in our discussion will manifest itself in Chapter \ref{Chapter:Cstaralgebra}, where it sets the stage for the $C^*-$algebra of a locally compact abelian group $G$.

The second example of a Banach $*-$algebra, which permeates the mathematical literature (Lemma 2.6.2 in \cite{DeitmarEchterhoff}, p. 50 or \cite{Dixmier}, p. 1 or \cite{Kaniuth}, p. 18), is the space $L^1(G),$ where $G$ is (as always in this paper) a locally compact abelian group. The \textit{convolution} defined by 
$$\forall_{x\in G}\ f\star g(x) := \int_G\ f(x-y)\cdot g(y)\ dy$$

\noindent
is the algebra multiplication in $L^1(G)$ and the involution is given by
$$\forall_{x\in G}\ f^*(x) := \overline{f(-x)}.$$

\noindent
The importance of this example is self-evident: after all, the main goal of the paper is to provide a Pego-like characterization of relatively compact families in $L^1(G).$

Before closing this chapter we introduce one more concept for future progress: a Banach $*-$algebra $\algebra$, which satisfies the \textit{$C^*-$property}: 
\begin{gather}
\forall_{f\in \algebra}\ \|f\cdot f^*\| = \|f\|^2,
\label{Cstarproperty}
\end{gather}

\noindent
is called a $C^*-$algebra (comp. \cite{BelfiDoran}, p. 5 or \cite{Davidson}, p. 1 or \cite{DeitmarEchterhoff}, p. 49 or \cite{Dixmier}, p. 8 or \cite{Kaniuth}, p. 66). Property (\ref{Cstarproperty}) is easily memorized as the functional analytic version of $z\cdot \overline{z} = |z|^2$ for $z\in\complex$ in complex analysis. $\BB(\Hilbert),$ which we mentioned earlier, is an example of a (noncommutative) $C^*-$algebra (comp. \cite{Dixmier}, p. 1 or \cite{DeitmarEchterhoff}, p. 49 or \cite{Kaniuth}, p. 66), but what about $L^1(G)$? If $L^1(G)$ were a (commutative) $C^*-$algebra, then by the celebrated Gelfand-Naimark theorem (comp. Theorem 7.1 in \cite{BelfiDoran}, p. 22 or Theorem I.3.1 in \cite{Davidson}, p. 7 or Theorem 2.6.7 in \cite{DeitmarEchterhoff}, p. 53 or Theorem 2.4.5 in \cite{Kaniuth}, p. 69), it would be isometrically $*-$isomorphic to $C_0(X),$ the space of continuous functions, which vanish at infinity ($X$ would be a certain locally compact space). This, in turn, would greatly simplify our work, as characterizing relatively compact families in $L^1(G)$ would then amount to characterizing relatively compact families in $C_0(X)$. Unfortunately, $L^1(G)$ is \textit{not} a $C^*-$algebra, unless $G$ is trivial (comp. Theorem 2.6.2 in \cite{DeitmarEchterhoff}, p. 50). This entails that if we want to write down Pego theorem for $L^1(G)$ then we need to overcome the obstacle of it not being a $C^*-$algebra. Resolving this issue lies at the heart of our next chapter.

\section{Convolution operator and the \texorpdfstring{$C^*-$}{Lg}algebra of a locally compact abelian group}
\label{Chapter:Cstaralgebra}

As foreshadowed in Chapter \ref{Chapter:Banachalgebras}, the current section will be devoted to the construction of a $C^*-$algebra of a locally compact abelian group $G$. This object will serve as the ``enhanced version'' of $L^1(G),$ which is not a $C^*-$algebra. 

To begin with, for fixed $f\in L^1(G)$ and $\varphi\in L^2(G)$ we consider the map
\begin{gather}
\phi \mapsto \int_G\ f(y)\cdot \langle \phi|T_y\varphi \rangle_2\ dy,
\label{linearfunctional}
\end{gather}

\noindent
where $T_y:L^2(G)\longrightarrow L^2(G)$ is the \textit{translation operator} $T_y\varphi(x) := \varphi(x-y)$ and $\langle\cdot|\cdot\rangle_2$ stands for the inner product in $L^2(G)$. On page 70 in \cite{DeitmarEchterhoff}, Deitmar and Echterhoff prove that the map (\ref{linearfunctional}) is a linear and bound functional on $L^2(G)$. By Riesz representation theorem (Theorem 4.11 in \cite{Brezis}, p. 97 or Theorem 6.52 in \cite{RenardyRogers}, p. 196) there exists a function $\Psi_{f,\varphi}\in L^2(G)$ such that 
\begin{gather}
\forall_{\phi \in L^2(G)}\ \langle \phi|\Psi_{f,\varphi}\rangle_2 = \int_G\ f(y)\cdot \langle \phi|T_y\varphi \rangle_2\ dy.
\label{Rieszcorollary}
\end{gather}

It is easy to observe that for every $\alpha_1,\alpha_2\in\complex$ and $\varphi_1,\varphi_2\in L^2(G)$ we have
$$\Psi_{f,\alpha_1\varphi_1+\alpha_2\varphi_2} = \alpha_1\Psi_{f,\varphi_1} + \alpha_2\Psi_{f,\varphi_2}.$$

\noindent
Furthermore, (\ref{Rieszcorollary}) leads to the following estimate (for details see \cite{DeitmarEchterhoff}, p. 70):
$$\|\Psi_{f,\varphi}\|_2 \leqslant \|f\|_1\cdot \|\varphi\|_2.$$

\noindent
Consequently, we may define a linear and bounded map $\Psi_f: L^2(G)\longrightarrow L^2(G)$ by $\Psi_f(\varphi) := \Psi_{f,\varphi}.$

Deitmar and Echterhoff go on to prove (Lemma 3.3.1 in \cite{DeitmarEchterhoff}, p. 70) that
\begin{gather}
\forall_{\varphi\in L^1(G)\cap L^2(G)}\ \Psi_f(\varphi) = f\star \varphi = \varphi \star f.
\label{Psiandconvolution}
\end{gather}

\noindent
Going off on a tangent for a brief moment, let us remark that we could define $\Psi_f$ in a slightly different manner than Deitmar and Echterhoff do. In this alternative approach (comp. \cite{Kaniuth}, p. 95), we would begin with \textit{defining} $\Psi_f(\varphi) := f\star \varphi$ for every $\varphi\in L^1(G)\cap L^2(G) -$ this definition is obviously inspired by the property (\ref{Psiandconvolution}) in Deitmar-Echterhoff's point of view. Since $L^1(G)\cap L^2(G)$ is dense in $L^2(G),$ we would subsequently extend the operator $\Psi_f:L^1(G)\cap L^2(G)\longrightarrow L^2(G)$ (by Theorem 1.6 in \cite{ReedSimon}, p. 9) to all of $L^2(G).$ This extension is unique, so our operator $\Psi_f$ would have no choice but to coincide with the one defined by Deitmar and Echterhoff. 

Going back to the main line of reasoning, we define the map $\Psi : L^1(G)\longrightarrow \BB(L^2(G))$ by the formula $\Psi(f) := \Psi_f.$ Lemma 3.3.2 in \cite{DeitmarEchterhoff}, p. 70 (or Theorem 2.7.7 in \cite{Kaniuth}, p. 95) proves that $\Psi$ is an injective algebra $*-$homomorphism between Banach $*-$algebras $L^1(G)$ and $\BB(L^2(G)).$ What more can we say about $\Psi$?


\begin{lemma}(known in the mathematical folklore)\\
If $f\in L^1(G)$ and $f\geqslant 0$ then $\|\Psi(f)\| = \|f\|_1.$
\label{normofPsiandfinL1}
\end{lemma}
\begin{proof}
First, let us remark that it is clear to see that $\Psi$ is a homeomorphism due to the general version of the open mapping theorem (comp. Theorem 2.11 in \cite{Rudin}, p. 48). However, we claim that even more is true, so a more detailed analysis is required. 

We fix $f\in L^1(G)$ and observe that 
$$\forall_{\phi\in L^1(G)\cap L^2(G)}\ \|\Psi_f(\phi)\|_2 \stackrel{(\ref{Psiandconvolution})}{=} \|f\star \phi\|_2 \leqslant \|f\|_1\cdot \|\phi\|_2,$$

\noindent
where the last inequality follows from Young's convolution inequality (comp. Corollary 20.14 in \cite{HewittRoss}, p. 293). We instantly conclude that $\|\Psi_f\|\leqslant \|f\|_1,$ which is the effortless part of the proof. Henceforth, we focus our attention on demonstrating that $\|\Psi_f\|\geqslant \|f\|_1.$

Since $f\in L^1(G)$ then $\supp(f)$ is contained in a $\sigma-$compact set (comp. Corolarry 1.3.5 in \cite{DeitmarEchterhoff}, p. 10). Hence
$$\supp(f)\subset \bigcup_{n=1}^{\infty}\ K_n,$$

\noindent
where $K_n \Subset G$ (meaning $K_n$ is compact) for every $n\in\naturals$. It is easy to see that without loss of generality, we may assume that 
\begin{itemize}
	\item $0\in K_n$ for every $n\in\naturals$ (adjoining one point does not change compactness of $K_n$),
	\item $K_n\subset K_{n+1}$ for every $n\in\naturals$ (this is a standard trick of substituting $D_n:= \bigcup_{m=1}^n\ K_m$ for $K_n$ if necessary),
	\item $K_n$ is symmetric for every $n\in\naturals$ (we may substitute $K_n\cup (-K_n)$ for $K_n$ if necessary).
\end{itemize}

\noindent
Consequently, we have that 
\begin{gather}
\forall_{n\geqslant 2}\ K_n - K_n \supset 0 - K_n = -K_n = K_n \supset K_{n-1}.
\label{inequalitythatweneedattheend}
\end{gather}

Furthermore, let us define a family of functions $\phi_n := \mu(K_n)^{-\frac{1}{2}}\cdot \mathds{1}_{K_n}$ for every $n\in\naturals,$ where $\mathds{1}_{K_n}$ is the characteristic function of $K_n.$ An easy calculation shows that $\|\phi_n\|_2 = 1$ for every $n\in\naturals.$  In order to conclude that $\|\Psi(f)\|\geqslant \|f\|_1,$ it is enough to demonstrate that 
$$\|\Psi_f(\phi_n)\| = \|f\star \phi_n\|_2 \geqslant \|f\|_1 - \eps$$

\noindent 
for a fixed $\eps>0$ and $n$ large enough. To this end we fix $\eps>0$ and choose $N\in\naturals$ such that 
\begin{gather}
\int_{K_N}\ f(x)\ dx \geqslant \|f\|_1 - \eps.
\label{choiceofNforepsilon}
\end{gather}

\noindent
Let us remark that the assumption ``$f\geqslant 0$'' plays a pivotal role in the choice of $N$: if $f$ could have been negative then it would be possible for the integral $\int_{K_N}\ f(x)\ dx$ to be negative and $\|f\|_1 - \eps$ to be positive. This would lead to contradiction. In other words, the choice of $N$ is possible due to the assumption that $f$ is non-negative.

Finally, we have 
\begin{gather*}
\forall_{n> N}\ \|f\star \phi_n\|_2^2 = \int_G\left|\int_G\ f(x-y)\cdot \phi_n(y)\ dy\right|^2\ dx = \frac{1}{\mu(K_n)}\int_G\ \left(\int_{K_n}\ f(x-y)\ dy\right)^2\ dx \\
\stackrel{y\mapsto x-y}{=} \frac{1}{\mu(K_n)}\int_G \left(\int_{x-K_n}\ f(y)\ dy\right)^2\ dx \geqslant \frac{1}{\mu(K_n)}\int_{K_n} \left(\int_{x-K_n}\ f(y)\ dy\right)^2\ dx \\
\stackrel{(\ref{inequalitythatweneedattheend})}{\geqslant} \frac{1}{\mu(K_n)}\int_{K_n} \left(\int_{K_N}\ f(y)\ dy\right)^2\ dx \stackrel{(\ref{choiceofNforepsilon})}{\geqslant} (\|f\|_1 - \eps)^2,
\end{gather*}

\noindent
which concludes the proof.
\end{proof}


	

The above lemma is elegant and charming, but the assumption ``$f\geqslant 0$'' limits its applicability to a certain degree. If the assumption of nonnegativity is removed, then the equality $\|\Psi(f)\| = \|f\|_1$ need not hold. In order to see this, let us recall that for every function $f\in L^1(G)$ we define its Fourier transform (comp. Chapter 1.7 in \cite{DeitmarEchterhoff}, p. 29 or Chapter 4.2 in \cite{Folland}, p. 93 or Chapter 2.8.4 in \cite{HavinNikolski}, p. 260 or Chapter 8 in \cite{HewittRoss2}, p. 209) with the formula
\begin{gather}
\forall_{\chi\in\widehat{G}}\ \widehat{f}(\chi) := \int_G\ f(x)\cdot \overline{\chi(x)}\ dx,
\label{Fouriertransform}
\end{gather}

\noindent
where $\widehat{G}$ is the dual group (comp. Chapter 7 in \cite{Deitmar}, p. 101 or Chapter 3.1 in \cite{DeitmarEchterhoff}, p. 63 or Chapter 4.1 in \cite{Folland}, p. 87 or Chapter 2.8.1 in \cite{HavinNikolski}, p. 255 or Chapter 6.23 in \cite{HewittRoss}, p. 355). We will be particularly interested in the dual group $\widehat{\integers}$, whose elements (characters) are the maps $\chi_{\alpha}:\integers\longrightarrow S^1$ given by
$$\chi_{\alpha}(k) := e^{2\pi i k\alpha},$$ 

\noindent
where $\alpha\in [0,1)$ (comp. Proposition 7.1.1 in \cite{Deitmar}, p. 101). If $[0,1)$ is treated as topological group with addition $\bmod 1$ then it can be shown (comp. Proposition 7.1.6 in \cite{Deitmar}, p. 106) that the map $\hh:\widehat{\integers}\longrightarrow ([0,1),+_{\bmod 1})$ given by
\begin{gather}
\hh(\chi_{\alpha}) := \alpha
\label{hisomorphism}
\end{gather}

\noindent
is a homeomorphic isomorphism. This fact allows for a convenient description of the Haar measure on $\widehat{\integers},$ due to the following result:

\begin{lemma}(known in the mathematical folklore)\\
Let $G_1$ and $G_2$ be locally compact (not necessarily abelian) groups and let $h:G_1\longrightarrow G_2$ be a homeomorphic isomorphism. If $\mu_1,\mu_2$ are (left) Haar measures on $G_1$ and $G_2$ respectively, then $\mu_1 = c\cdot\mu_2\circ h$ for some constant $c>0.$ 
\label{Haarmeasuresandhomeomorphiciso}
\end{lemma}
\begin{proof}
Since a Haar measure is unique up to a positive constant (comp. Theorem 2.20 in \cite{Folland}, p. 39 or Theorem 3.3.2 in \cite{ReiterStegman}, p. 83) then it suffices to check that $\mu_2\circ h$ is a Haar measure on $G_1$. First of all, it is a nontrivial map since $\mu_2\circ h(G_1) = \mu_2(G_2) > 0.$ Obviously, we have $\mu_2\circ h(\emptyset) = \mu_2(\emptyset) = 0.$ Next, if $(A_n)_{n\in\naturals}\subset G_1$ is a sequence of pairwise disjoint Borel-measurable sets then the sequence $(h(A_n))_{n\in\naturals}\subset G_2$ is also pairwise disjoint and we have 
$$\mu_2\circ h\left(\bigcup_{n\in\naturals}\ A_n\right) = \mu_2\left(\bigcup_{n\in\naturals}\ h(A_n)\right) = \sum_{n\in\naturals}\ \mu_2\circ h(A_n).$$

\noindent
This demonstrates that $\mu_2\circ h$ is a nontrivial Borel measure. 

If $K$ is a compact subset of $G_1$ then $h(K)$ is compact in $G_2$ and $\mu_2\circ h(K) < \infty.$ Furthermore, to see that $\mu_2\circ h$ is inner regular we take an open set $U\subset G_1$ (then $h(U)$ is open in $G_2$) and observe that 
\begin{equation*}
\begin{split}
\mu_2\circ h(U) &= \sup \left\{\mu_2(D)\ :\ D\subset h(U), D\Subset G_2 \right\} \\
&= \sup \left\{\mu_2\circ h(h^{-1}(D))\ :\ h\circ h^{-1}(D)\subset h(U), D\Subset G_2\right\}\\
&= \sup \left\{\mu_2\circ h(h^{-1}(D))\ :\ h^{-1}(D)\subset U, D\Subset G_2\right\} \\
&= \sup \left\{\mu_2\circ h(K)\ :\ K\subset U, K\Subset G_1\right\}.
\end{split}
\end{equation*}

\noindent
Analogously, if $A\subset G_1$ is a Borel-measurable set then $h(A)$ is Borel-measurable in $G_2$ and
\begin{equation*}
\begin{split}
\mu_2\circ h(A) &= \inf \left\{\mu_2(U)\ :\ h(A)\subset U, U\ \text{is open in}\ G_2\right\} \\
&= \inf \left\{\mu_2\circ h(h^{-1}(U))\ :\ h(A)\subset h(h^{-1}(U)), U\ \text{is open in}\ G_2\right\}\\
&= \inf \left\{\mu_2\circ h(h^{-1}(U))\ :\ A\subset h^{-1}(U), U\ \text{is open in}\ G_2\right\}\\
&= \inf \left\{\mu_2\circ h(V)\ :\ A\subset V, V\ \text{is open in}\ G_2\right\},
\end{split}
\end{equation*}

\noindent
which establishes that $\mu_2\circ h$ is outer regular. Last but not least, if $x\in G_1$ and $A\subset G_1$ is a Borel-measurable set then
$$\mu_2\circ h(x\cdot A) = \mu_2\left(h(x)\cdot h(A)\right) = \mu_2\circ h(A),$$

\noindent
due to left-invariance of $\mu_2$ (an the fact that $h$ is an isomorphism). This concludes the proof.
\end{proof}

We now apply Lemma \ref{Haarmeasuresandhomeomorphiciso} to describe the Haar measure and Haar integral on $\widehat{\integers}.$ In the corollary below, we use the term ``\textit{normalized} Haar measure'', which means that the measure of the whole (compact) group is equal to $1$.

\begin{cor}(known in the mathematical folklore)\\
The normalized Haar measure $\mu_{\widehat{\integers}}$ on $\widehat{\integers}$ satisfies
\begin{gather}
\mu_{\widehat{\integers}} = \lambda|_{[0,1)}\circ \hh,
\label{Haarmeasureonwidehatintegers}
\end{gather}

\noindent
where $\lambda$ is the classical Lebesgue measure on $\reals$. Consequently, if $F\in L^1(\widehat{\integers})$ then
\begin{gather}
\int_{\widehat{\integers}}\ F(\chi)\ d\chi = \int_0^1\ F\circ\hh^{-1}(\alpha)\ d\alpha.
\label{howtocalculateintegraloverwidehatintegers}
\end{gather}
\end{cor}
\begin{proof}
First of all, we remark that $([0,1), +_{\bmod 1})$ is a compact group, since it is homeomorphic to the circle group $S^1$ (or the quotient group $\reals/\integers$ for that matter). Furthermore, $\lambda|_{[0,1)}$ is the normalized Haar measure on $([0,1), +_{\bmod 1}).$ By Lemma \ref{Haarmeasuresandhomeomorphiciso} we have that $\lambda|_{[0,1)}\circ \hh$ is the normalized Haar measure on $\widehat{\integers}.$ This demonstrates the equality (\ref{Haarmeasureonwidehatintegers}).

Regarding the second part of the theorem, we begin by observing that if $F = \mathds{1}_A$ for some measurable $A\subset \widehat{\integers},$ then
\begin{gather*}
\int_{\widehat{\integers}}\ F(\chi)\ d\chi = \int_{\widehat{\integers}}\ \mathds{1}_A(\chi)\ d\chi = \mu_{\widehat{\integers}}(A) \stackrel{(\ref{Haarmeasureonwidehatintegers})}{=} \lambda|_{[0,1)}\circ \hh(A) = \int_0^1\ \mathds{1}_{\hh(A)}(\alpha)\ d\alpha\\
=\int_0^1\ \mathds{1}_A(\hh^{-1}(\alpha))\ d\alpha = \int_0^1\ F\circ\hh^{-1}(\alpha)\ d\alpha,
\end{gather*}

\noindent
which is congruent with (\ref{howtocalculateintegraloverwidehatintegers}). The rest of the argument follows a classical technique. We take the liberty of not transcribing the whole reasoning in detail, but provide just a simple sketch: since (\ref{howtocalculateintegraloverwidehatintegers}) is true for the characteristic functions (as we have just demonstrated) then it is also true for simple functions (with non-negative coefficients). Moreover, given a non-negative function $F\in L^1(\widehat{\integers})$ we approximate it by a nondecreasing sequence of non-negative simple functions (comp. Theorem 2.10 in \cite{FollandRealAnalysis}, p. 47) and (\ref{howtocalculateintegraloverwidehatintegers}) for such $F$ follows from the monotone convergence theorem (comp. Theorem 2.14 in \cite{FollandRealAnalysis}, p. 50). Next, any real-valued integrable function can be represented as a difference of two non-negative integrable functions, so (\ref{howtocalculateintegraloverwidehatintegers}) holds for any real-valued $F\in L^1(\widehat{\integers}).$ Last but not least, every complex-valued $F\in L^1(\widehat{\integers})$ can be decomposed as $F_1+iF_2,$ where $F_1,F_2\in L^1(\widehat{\integers})$ are real-valued. This implies that (\ref{howtocalculateintegraloverwidehatintegers}) holds true for every function in $L^1(\widehat{\integers}).$  
\end{proof}

In the next theorem we calculate the norm $\|\Psi_f\|$ for $f\in L^1(\integers).$ This computation will prove to be invaluable in Corollary \ref{Psinotisometry}, in which we demonstrate the existence of a function $g\in L^1(\integers)$ such that $\|\Psi_g\| < \|g\|_1.$

The theorem below bears a dim resemblance to Corollary 1.2 in \cite{GohbergGoldbergKaashoek}, p. 217. The result is also mentioned (without any reference or proof) in \cite{Kovrizhkin}. However, as we were unable to find a reference with proof for this result, we have taken the liberty of proving it ourselves.

\begin{thm}
If $G=\integers$ then 
$$\forall_{f\in L^1(\integers)}\ \|\Psi_f\| = \|\widehat{f}\|_{\infty},$$

\noindent
where $\|\cdot\|_{\infty}$ denotes the supremum norm. 
\label{normofconvolutionoperator}
\end{thm}
\begin{proof}
We fix $f\in L^1(\integers)$ and observe that
\begin{gather}
\|\Psi_f\| = \sup_{\substack{\phi\in C_c(\integers)\\ \|\phi\|_2 = 1}}\ \|\Psi_f(\phi)\|_2 = \sup_{\substack{\phi\in C_c(\integers)\\ \|\phi\|_2 = 1}}\ \|f\star \phi\|_2 
= \sup_{\substack{\phi\in C_c(\integers)\\ \|\phi\|_2 = 1}}\ \|\widehat{f\star\phi}\|_2 = \sup_{\substack{\phi\in C_c(\integers)\\ \|\phi\|_2 = 1}}\ \|\widehat{f}\cdot\widehat{\phi}\|_2, 
\label{calculatethenormpsif}
\end{gather}

\noindent
where we have employed both the Plancherel theorem (comp. Theorem 3.4.8 in \cite{DeitmarEchterhoff}, p. 77 or Theorem 4.25 in \cite{Folland}, p. 99 or Theorem 4.4.1 in \cite{Kaniuth}, p. 219 ) and the property $\widehat{f\star\phi} = \widehat{f}\cdot\widehat{\phi}$ (comp. Theorem 8.3.1 in \cite{Deitmar}, p. 120 or Lemma 1.7.2 in \cite{DeitmarEchterhoff}, p. 30). Furthermore, we have
\begin{gather*}
\forall_{\substack{\phi\in C_c(\integers)\\ \|\phi\|_2 = 1}}\|\widehat{f}\cdot\widehat{\phi}\|_2^2 = \int_{\widehat{\integers}}\ |\widehat{f}(\chi)\cdot\widehat{\phi}(\chi)|^2\ d\chi\leqslant \|\widehat{f}\|_{\infty}^2 \cdot \int_{\widehat{\integers}}\ |\widehat{\phi}(\chi)|^2\ d\chi = \|\widehat{f}\|_{\infty}^2\cdot \|\widehat{\phi}\|_2^2 = \|\widehat{f}\|_{\infty}^2\cdot \|\phi\|_2^2 = \|\widehat{f}\|_{\infty}^2,
\end{gather*}

\noindent
which (coupled with (\ref{calculatethenormpsif})) establishes that $\|\Psi_f\|\leqslant \|\widehat{f}\|_{\infty}.$

For the converse, we fix $\alpha_*\in[0,1)$ and define a sequence of functions $\Phi_n:\widehat{\integers}\longrightarrow\complex$ by
\begin{gather*}
\Phi_n(\chi_{\alpha}) := \left\{\begin{array}{cl}
\sqrt{n} & \text{ if}\ |\alpha-\alpha_*| < \frac{1}{2n},\\
0 & \text{ otherwise.}
\end{array}\right.
\end{gather*}

\noindent
Naturally $\Phi_n\in L^2(\widehat{\integers})$ and we have
\begin{gather}
\|\widehat{f}\cdot \Phi_n\|_2^2 = \int_{\widehat{\integers}}\ |\widehat{f}(\chi)\cdot \Phi_n(\chi)|^2\ d\chi \stackrel{(\ref{howtocalculateintegraloverwidehatintegers})}{=} \int_0^1\ |\widehat{f}\circ \hh^{-1}(\alpha)|^2\cdot |\Phi_n\circ \hh^{-1}(\alpha)|^2\ d\alpha \geqslant n\int_{\alpha_*-\frac{1}{2n}}^{\alpha_*+\frac{1}{2n}}\ |\widehat{f}\circ \hh^{-1}(\alpha)|\ d\alpha.
\label{prepareforHospital}
\end{gather}

\noindent
By a standard application of de l'H\^ospital's rule we conclude that 
\begin{gather}
\limsup_{n\rightarrow\infty}\ \|\widehat{f}\cdot \Phi_n\|_2 \stackrel{(\ref{prepareforHospital})}{\geqslant} |\widehat{f}\circ \hh^{-1}(\alpha_*)| = |\widehat{f}(\chi_{\alpha_*})|.
\label{applicationofHospital} 
\end{gather}

\noindent
By density of $C_c(\integers)$ in $L^2(\integers)$ (comp. Proposition 7.9 in \cite{FollandRealAnalysis}, p. 217) we pick a sequence $(\phi_n)_{n\in\naturals}\subset C_c(\integers)$ such that 
\begin{gather}
\|\phi_n - \widecheck{\Phi_n}\|_2 \stackrel{n\rightarrow\infty}{\longrightarrow} 0,
\label{choiceofphin}
\end{gather}

\noindent
where $\widecheck{\Phi_n}\in L^2(\integers)$ denotes the inverse Fourier transform of $\Phi_n.$ In particular, (\ref{choiceofphin}) means that $\|\phi_n\|\stackrel{n\rightarrow\infty}{\longrightarrow} 1$, since (again by Plancherel theorem) $\|\widecheck{\Phi_n}\|_2 = \|\Phi_n\|_2 = 1$ for every $n\in\naturals.$

Finally, we have
\begin{equation*}
\begin{split}
\|\Psi_f\| &\geqslant \limsup_{n\rightarrow \infty}\ \frac{\|\Psi_f(\phi_n)\|_2}{\|\phi_n\|_2} = \limsup_{n\rightarrow\infty}\ \|f\star\phi_n\|_2 = \limsup_{n\rightarrow\infty}\ \|\widehat{f\star\phi_n}\|_2= \limsup_{n\rightarrow\infty}\ \|\widehat{f}\cdot\widehat{\phi_n}\|_2 \\
&\geqslant \limsup_{n\rightarrow\infty} \bigg( \|\widehat{f}\cdot\Phi_n\|_2 - \|\widehat{f}\cdot (\widehat{\phi_n} - \Phi_n)\|_2\bigg) \geqslant \limsup_{n\rightarrow\infty} \bigg( \|\widehat{f}\cdot\Phi_n\|_2 - \|\widehat{f}\|_{\infty} \cdot \|\widehat{\phi_n} - \Phi_n\|_2\bigg) \\
&= \limsup_{n\rightarrow\infty} \bigg( \|\widehat{f}\cdot\Phi_n\|_2 - \|\widehat{f}\|_{\infty} \cdot \|\phi_n - \widecheck{\Phi_n}\|_2\bigg) \stackrel{(\ref{choiceofphin})}{=} \limsup_{n\rightarrow\infty}\  \|\widehat{f}\cdot\Phi_n\|_2 \stackrel{(\ref{applicationofHospital})}{\geqslant} |\widehat{f}(\chi_{\alpha_*})|.
\end{split}
\end{equation*}

\noindent
As the choice of $\alpha_*\in[0,1)$ (equivalently $\chi_{\alpha_*}\in\widehat{\integers}$) was arbitrary, we are done.
\end{proof}

\begin{cor}
There exists a function $g\in L^1(\integers)$ such that $\|\widehat{g}\|_{\infty} < \|g\|_1,$ and consequently $\|\Psi(g)\| < \|g\|_1.$ 
\label{Psinotisometry}
\end{cor}
\begin{proof}
We put
\begin{gather*}
g(n) := \left\{\begin{array}{cl}
1 & \text{ for }n=0,1 \\
-1 & \text{ for }n=2 \\
0 & \text{ otherwise}. \\
\end{array}\right.
\end{gather*}

\noindent
Obviously $g\in L^1(\integers)$ with $\|g\|_1 = 3.$ Subsequently, we perform the following calculation:
\begin{gather*}
\forall_{y\in[0,1]}\ |\widehat{g}(y)|^2 = \left|\sum_{n\in\integers}\ g(n)\cdot \exp(-2\pi iny)\right|^2 =  \left|1 + \exp(-2\pi iy) - \exp(-4\pi iy)\right|^2\\
= \left|1+\cos(2\pi y)-\cos(4\pi y) - i\big(\sin(2\pi y) - \sin(4\pi y)\big)\right|^2 \\
= \big(1+\cos(2\pi y)-\cos(4\pi y)\big)^2 + \big(\sin(2\pi y) - \sin(4\pi y)\big)^2\\
= 3-2\cos(4\pi y) \leqslant 5, 
\end{gather*}

\noindent
where we have ommitted a number of trivial applications of trigonometric identities. The above estimate demonstrates that 
$$\|\widehat{g}\|_{\infty} \leqslant \sqrt{5} < 3 = \|g\|_1,$$

\noindent
which concludes the proof.
\end{proof}

A rather sad upshot of Corollary \ref{Psinotisometry} is that on some groups (like $\integers$ as we demonstrated), the operator $\Psi$ is not an isometry. Nevertheless, we use this map to define the \textit{group $C^*-$algebra} as the norm$-$closure of $\Psi(L^1(G))$ inside the $C^*-$algebra $\BB(L^2(G)).$ Since $\Psi$ is an injective algebra $*-$homomorphism between $L^1(G)$ and $\BB(L^2(G))$ (comp. Lemma 3.3.2 in \cite{DeitmarEchterhoff}, p. 70 or Theorem 2.7.7 in \cite{Kaniuth}, p. 95) then $\Psi(L^1(G))$ is a commutative, normed $*-$subalgebra of $\BB(L^2(G)),$ which satisfies the $C^*-$property. Taking the closure simply ``completes'' $\Psi(L^1(G))$ in $\BB(L^2(G))$, turning it into a commutative $C^*-$algebra, which we denote by $C^*(G).$




\section{Arzel\`a, Ascoli and Pego}
\label{Chapter:ArzelaAscoliPego}

At last, we have arrived at the main chapter of the paper. We already have the proper ``canvas'' to work with, namely the $C^*-$algebra $C^*(G),$ where $G$ is (as always) a locally compact abelian group. As we demonstrate in the first theorem of the current chapter, $C^*(G)$ is \textit{de facto} the space of continuous functions (on the dual group $\widehat{G}$), which vanish at infinity. Therefore, employing a correct version of the Arzel\`a-Ascoli theorem we should be able to characterize relatively compact families in $C^*(G)$. This approach will bear fruit in the form of Pego theorem for $L^1(G),$ i.e. a characterization of relatively compact families in $L^1(G)$ via the Fourier transform. 


\begin{thm}
There exists an isometric $*-$isomorphism $\Omega$ between $C_0(\widehat{G})$ and $C^*(G).$
\label{C0widehatGCstarG}
\end{thm}
\begin{proof}
By Theorem 3.3.3 in \cite{DeitmarEchterhoff}, p. 71 we know that the map 
\begin{gather}
\Psi_{\Delta} : \Delta_{C^*(G)}\longrightarrow \Delta_{L^1(G)}\hspace{0.4cm} \text{given by}\hspace{0.4cm} \Psi_{\Delta}(m) := m\circ \Psi
\label{definitionPsiDelta}
\end{gather}

\noindent
is a homeomorphism. We easily derive an explicit formula for the inverse of $\Psi_{\Delta}:$
\begin{gather}
\forall_{m\in \Delta_{L^1(G)}}\ \Psi_{\Delta}^{-1}(m) = m\circ \Psi^{-1}.
\label{inversePsi}
\end{gather} 

\noindent
By Theorem 3.2.1 in \cite{DeitmarEchterhoff}, p. 67 the map $M:\widehat{G}\longrightarrow \Delta_{L^1(G)}$ given by
\begin{equation}
M(\chi) := m_{\chi},\hspace{0.4cm} \text{where}\hspace{0.4cm} m_{\chi}(f) := \widehat{f}(\chi) \stackrel{(\ref{Fouriertransform})}{=} \int_G\ f(x)\cdot \overline{\chi(x)}\ dx,
\label{definitionM}
\end{equation}

\noindent
is a homeomorphism. Consequently, the map $M^{-1}\circ \Psi_{\Delta}$ is a homeomorphism between $\Delta_{C^*(G)}$ and $\widehat{G}.$ As in Corollary 2.2.13 in \cite{Kaniuth}, p. 57 we establish that the map 
\begin{gather}
\Theta: C_0(\widehat{G})\longrightarrow C_0\left(\Delta_{C^*(G)}\right)\hspace{0.4cm} \text{given by} \hspace{0.4cm} \Theta(F) := F \circ M^{-1}\circ \Psi_{\Delta}
\label{definitionTheta}
\end{gather}

\noindent
is an isometric algebra isomorphism. Furthermore, we easily note that 
$$\Theta(F^*) = F^*\circ M^{-1}\circ \Psi_{\Delta} = \overline{F\circ M^{-1}\circ \Psi_{\Delta}} = \Theta(F)^*,$$

\noindent
so $\Theta$ is in fact an isometric $*-$isomorphism. Again, we easily derive an explicit formula for the inverse of $\Theta:$
\begin{gather}
\forall_{f\in C_0\left(\Delta_{C^*(G)}\right)}\ \Theta^{-1}(f) = f\circ \Psi_{\Delta}^{-1}\circ M.
\label{inverseTheta}
\end{gather}

By the celebrated Gelfand-Naimark theorem (comp. Theorem 7.1 in \cite{BelfiDoran}, p. 22 or Theorem I.3.1 in \cite{Davidson}, p. 7 or Theorem 2.6.7 in \cite{DeitmarEchterhoff}, p. 53 or Theorem 2.4.5 in \cite{Kaniuth}, p. 69), the Gelfand transform $\Gamma: C^*(G)\longrightarrow C_0(\Delta_{C^*(G)})$ given by
$$\Gamma(f) := \widehat{f}, \hspace{0.4cm}\text{where}\hspace{0.4cm} \widehat{f}(m) := m(f),$$

\noindent
is an isometric $*-$isomorphism. In conclusion, the map $\Omega : C_0(\widehat{G})\longrightarrow C^*(G)$ given by
\begin{gather}
\Omega:= \Gamma^{-1}\circ \Theta
\label{definitionofOmega}
\end{gather}

\noindent
is an isometric $*$-isomorphism.
\end{proof}

The relation between the newly-introduced map $\Omega$ and the convolution operator $\Psi$ (the injective $*-$homomorphism constructed in Chapter \ref{Chapter:Cstaralgebra}) is pictured on the diagram below:
$$L^1(G) \stackrel{\Psi}{\longrightarrow} \Psi(L^1(G))\subset C^*(G) \stackrel{\Omega^{-1}}{\longrightarrow} C_0(\widehat{G}),$$

\noindent
Our approach to characterize relatively compact families in $L^1(G)$ is rather straightforward $-$ the key idea is that 
\begin{center}
\textit{relative compactness of $\Ffamily\subset L^1(G)$ is equivalent (because $\Psi$ and $\Omega^{-1}$ are homeomorphisms) to relative compactness of $\Omega^{-1}\circ \Psi(\Ffamily)\subset C_0(\widehat{G})$.}
\end{center}

\noindent
Hence, characterizing relatively compact families in $L^1(G)$ should not be more difficult than finding the Arzel\`a-Ascoli theorem for $C_0(\widehat{G})$. 

The classic version of Arzel\`a-Ascoli theorem can be found in countless sources: Theorem 4.25 in \cite{Brezis}, p. 111 or Theorem 23.2 in \cite{Choquet}, p. 99 or Theorem 6.3.1 in \cite{Dixmiertopology}, p. 69 or Theorem 4.43 in \cite{FollandRealAnalysis}, p. 137 or Corollary 10.49 in \cite{Knapp}, p. 479 or Theorem 11.28 in \cite{Rudinrealandcomplex}, p. 245 or Theorem A5 in \cite{Rudin}, p. 394 etc.

\begin{thm}(classic version of Arzel\`a-Ascoli theorem)\\
Let $X$ be a compact (Hausdorff) space. The family $\Ffamily\subset C(X)$ is relatively compact if and only if 
\begin{itemize}
	\item $\Ffamily$ is pointwise bounded, i.e. for every $x\in X$ there exists $M_x>0$ such that 
	$$\forall_{f\in\Ffamily}\ |f(x)| < M_x,$$
	\item $\Ffamily$ is equicontinuous at every point, i.e. for every $x\in X$ and $\eps>0$ there exists an open neighbourhood $U_x$ of $x$ such that 
	$$\forall_{\substack{y\in U_x,\\ f\in\Ffamily}}\ |f(y)- f(x)| < \eps.$$
\end{itemize}
\label{classicalAA}
\end{thm}

However, in our case the domain $X = \widehat{G}$ need not be compact. In fact, if $\widehat{G}$ is compact then $G$ must be finite (comp. Proposition 4.35 in \cite{Folland}, p. 103), which we do not assume. Despite best efforts, we were unable to find a good reference for a theorem characterizing relatively compact families in $C_0(X),$ where $X$ is a locally compact group. Even worse, we came across a version of Arzel\`a-Ascoli theorem that is palpably wrong!

\begin{thm}(comp. Theorem A.1.4 in \cite{Kaniuth}, p. 320)\\
Let $X$ be a locally compact Hausdorff space and $\Ffamily\subset C_0(X).$ Suppose that $\Ffamily$ satisfies the following two conditions:
\begin{itemize}
	\item $\Ffamily$ is pointwise bounded,
	\item $\Ffamily$ is equicontinuous at every point.
\end{itemize}

\noindent
Then $\Ffamily$ is relatively compact in $(C_0(X),\|\cdot\|_{\infty})$.
\label{wrongtheorem}
\end{thm}

If this theorem were true then it would work in particular for $X = \integers$. We observe that for such a choice of $X$, the equicontinuity condition becomes obsolete $-$ every family $\Ffamily\subset C_0(\integers)$ is equicontinuous since we may always pick $U_x = \{x\}.$ Thus it suffices to consider a sequence of characteristic functions $\Ffamily = \left(\mathds{1}_{\{n\}}\right)_{n\in\naturals},$ which is obviously pointwise bounded. However, this sequence does not contain any convergent subsequence, so $\Ffamily$ cannot be relatively compact (contrary to what Theorem \ref{wrongtheorem} implies)! This demonstrates that Theorem \ref{wrongtheorem} cannot be true.

In the pursuit of the proper characterization of relatively compact families in $C_0(X)$ we came across Exercise 17 on page 182 in John B. Conway's ``A Course in Functional Analysis'' (comp. \cite{Conway}), which is likely the closest to what we need. Unfortunately, Conway does not provide the proof, leaving this task to the reader. Discouraged by the fruitless search for a proper reference, we have decided to provide a full proof of the Arzel\`a-Ascoli theorem for $C_0(X)$ ourselves. Our approach is inspired by Chapter 7 in Kelley's ``General Topology'' (comp. \cite{Kelley}) as well as Chapters 46 and 47 in Munkres' ``Topology'' (comp. \cite{Munkres}). In their monographs, Kelley and Munkres employ the technique of ``\textit{collapsing topologies}'', which we briefly summarize. There are three main topologies on $C_0(X):$
\begin{itemize}
	\item \textit{Topology of pointwise convergence} $\tau_{pc}.$ The family
	$$\{f\in C_0(X)\ :\ |f(x) - f_*(x)| < \eps\}_{x\in X, f_*\in C_0(X), \eps>0}$$
	
	\noindent
	forms a subbase for this topology (comp. \cite{Munkres}, p. 281).
	\item \textit{Topology of uniform convergence on compact sets} $\tau_{ucc}.$ The family 
	$$\{f\in C_0(X)\ :\ \forall_{x\in K}\ |f(x)-f_*(x)| < \eps\}_{K\Subset X, f_*\in C_0(X), \eps>0}$$
	
	\noindent
	forms a base for this topology (comp. \cite{Munkres}, p. 283).
	
	\item \textit{Topology of uniform convergence} $\tau_{uc}.$ The family 
	$$\{f\in C_0(X)\ :\ \forall_{x\in X}\ |f(x)-f_*(x)| < \eps\}_{f_*\in C_0(X), \eps>0}$$
	
	\noindent
	forms a base for this topology.
\end{itemize}

If we consider $C_0(X)$ with the topology of pointwise convergence $\tau_{pc},$ then the Arzel\`a-Ascoli theorem is equivalent to Tychonoff's theorem (comp. Theorem 1 in \cite{Kelley}, p. 218). The crucial observation is that if $\Ffamily\subset C_0(X)$ is equicontinuous, then $\tau_{pc}$ and $\tau_{ucc}$ coincide on $\Ffamily$ (comp. Theorem 15 in \cite{Kelley}, p. 232). We say, rather informally, that the topologies ``collapse onto each other'' on $\Ffamily$. This line of reasoning lies at the heart of Arzel\`a-Ascoli theorem for $C_0(X)$ with the topology $\tau_{ucc}.$ 

At this point it is at least plausible to believe that there exists some condition on the family $\Ffamily\subset C_0(X)$ under which the topologies $\tau_{ucc}$ and $\tau_{uc}$ coincide. We simply need to work out what this condition might be.

\begin{lemma}
Let $X$ be a locally compact space. If the family $\Ffamily\subset C_0(X)$ is equivanishing, i.e.
\begin{gather}
\forall_{\eps>0}\ \exists_{K\Subset X}\ \forall_{f\in\Ffamily}\ \sup_{x\in X\backslash K}\ |f(x)| < \eps,
\label{assumptionequivanishing}
\end{gather}

\noindent
then the topologies $\tau_{ucc}$ and $\tau_{uc}$ coincide on $\Ffamily,$ i.e. $\tau_{co}|_{\Ffamily} = \tau_{uc}|_{\Ffamily}.$ 
\label{equivanishingandtopologycollapse}
\end{lemma}
\begin{proof}
By Theorem 46.7 in \cite{Munkres}, p. 285 we know that $\tau_{ucc}\subset \tau_{uc}$ so we only need to prove the reverse inclusion holds on $\Ffamily$. We fix $f_*\in C_0(X),\ \eps>0$ and define a $\tau_{uc}-$open set
$$U_{f_*,\eps} := \big\{f\in\Ffamily\ :\ \forall_{x\in X}\ |f(x) - f_*(x)| < \eps\big\}.$$

\noindent
Our task is as follows: for every $g\in U_{f_*,\eps}$ we need to define $\tau_{ucc}-$open set
$$V_{K, g, \delta} := \big\{f\in\Ffamily\ :\ \forall_{x\in K}\ |f(x)-g(x)| < \delta\big\}$$

\noindent
such that $V_{K, g, \delta}\subset U_{f_*,\eps}.$

We fix $g\in U_{f_*,\eps}$ and define 
\begin{gather}
d:= \|g-f_*\|_{\infty}.
\label{wedefined}
\end{gather} 

\noindent
We pick $\delta>0$ such that 
\begin{gather}
3\delta + d < \eps.
\label{choiceofdelta}
\end{gather}

\noindent
By (\ref{assumptionequivanishing}) there exists $K\Subset X$ such that 
\begin{gather}
\forall_{f\in\Ffamily}\ \sup_{x\in X\backslash K}\ |f(x)| < \delta.
\label{wechooseK}
\end{gather}

\noindent
Finally, if $f\in V_{K,g,\delta}$ then 
\begin{gather*}
\|f-f_*\|_{\infty} \leqslant \|f-g\|_{\infty} + \|g-f_*\|_{\infty} \stackrel{(\ref{wedefined})}{\leqslant} \sup_{x\in K}\ |f(x)-g(x)| + \sup_{x\in X\backslash K}\ |f(x)- g(x)|+ d \\
\leqslant \delta + \sup_{x\in X\backslash K}\ |f(x)| + \sup_{x\in X\backslash K}\ |g(x)| + d \stackrel{(\ref{wechooseK})}{\leqslant} 3\delta + d \stackrel{(\ref{choiceofdelta})}{<} \eps,
\end{gather*}

\noindent
which concludes the proof. 
\end{proof}

With Lemma \ref{equivanishingandtopologycollapse} at our disposal we are ready to characterize relatively compact families in $C_0(X).$

\begin{thm}(Arzel\`a-Ascoli theorem for $C_0(X)$)\\
Let $X$ be a locally compact space. A family $\Ffamily\subset C_0(X)$ is relatively compact with respect to $\|\cdot\|_{\infty}-$norm (equivalently $\tau_{uc}-$topology) if and only if 
\begin{description}
	\item[\hspace{0.4cm} (AA1)] $\Ffamily$ is pointwise bounded,
	\item[\hspace{0.4cm} (AA2)] $\Ffamily$ is equicontinuous at every point,
	\item[\hspace{0.4cm} (AA3)] $\Ffamily$ is equivanishing. 
\end{description}
\label{AAforC0X}
\end{thm}
\begin{proof}
For the ``\textit{if}'' part, we observe that if $\Ffamily\subset C_0(X)$ is both pointwise bounded and equicontinuous at every point then it is relatively $\tau_{ucc}-$compact (comp. Theorem 17 in \cite{Kelley}, p. 233 or Theorem 47.1 in \cite{Munkres}, p. 290). Moreover, since $\tau_{ucc}$ and $\tau_{uc}$ coincide on $\Ffamily$ (due to condition (\textbf{AA3}) and Lemma \ref{equivanishingandtopologycollapse}) then $\Ffamily$ is also relatively $\tau_{uc}-$compact.

As far as the ``\textit{only if}'' part is concerned, Theorem 17 in \cite{Kelley}, p. 233 (or Theorem 47.1 in \cite{Munkres}, p. 290) demonstrates that a relatively $\tau_{uc}-$compact family $\Ffamily\subset C_0(X)$ is both pointwise bounded and equicontinuous at every point. Consequently, it remains to prove that $\Ffamily$ is equivanishing. We fix $\eps>0$ and (due to relative $\tau_{uc}-$compactness) we pick an $\eps-$net $(f_n)_{n=1}^N,$ i.e. a finite sequence of functions in $\Ffamily$ such that 
\begin{gather}
\forall_{f\in\Ffamily}\ \exists_{n=1,\ldots,N}\ \|f-f_n\|_{\infty} < \eps.
\label{choiceofepsnet}
\end{gather}

\noindent
Next, we let $K\Subset G$ be such that 
\begin{gather}
\forall_{n=1,\ldots,N}\ \sup_{x\in X\backslash K}\ |f_n(x)| < \eps.
\label{choiceofKcompact}
\end{gather}

\noindent
Finally, we have
\begin{gather*}
\forall_{\substack{n=1,\ldots,N\\ f\in\Ffamily}}\ \sup_{x\in X\backslash K}\ |f(x)| \leqslant \sup_{x\in X\backslash K}\ |f(x) - f_n(x)| + \sup_{x\in X\backslash K}\ |f_n(x)| \stackrel{(\ref{choiceofKcompact})}{<} \|f-f_n\|_{\infty} + \eps.
\end{gather*}

\noindent
Due to (\ref{choiceofepsnet}) and the arbitrary choice of $\eps>0$ we conclude that $\Ffamily$ is equivanishing, which ends the proof.
\end{proof}

Before we conclude the paper with Pego theorem for $L^1(G)$ let us introduce one final improvement to Theorem \ref{AAforC0X}. The idea for this innovation dates back to the work of Vladimir N. Sudakov $-$ a modern exposition of his techniques can be found in \cite{GorkaRafeiro} or \cite{HancheOlsenHoldenMalinnikova}. 

We need the following technical lemma:

\begin{lemma}
Let $X$ be a locally compact (not necessarily abelian) group and let $K\Subset X.$ If a family $\Ffamily \subset C_0(X)$ is equicontinuous at every point, then for every $\eps>0$ there exists an open neighbourhood $V_e$ of the neutral element $e\in X$ such that 
$$\forall_{\substack{y\in V_e\\ f\in\Ffamily}}\ \sup_{x\in K}\ |f(yx) - f(x)| < \eps.$$
\label{technicallemmasudakov}
\end{lemma}
\begin{proof}
Fix $\eps>0.$ For every $x\in K$ we choose an open neighbourhood $U_x$ of the neutral element (note that $U_x$ need not be a neighbourhood of $x$) such that 
\begin{gather}
\forall_{\substack{y\in U_xx\\ f\in\Ffamily}}\ |f(y)- f(x)| < \eps,
\label{equicontSudakov}
\end{gather}

\noindent
according to the definition of equicontinuity in Theorem \ref{classicalAA}. Furthermore, for every $x\in K$ we let $V_x$ be an open neighbourhood of the identity element such that $V_x^2\subset U_x.$ The family $(V_xx)_{x\in K}$ is an open cover of $K$ so (by the compactness of $K$) we choose a finite subcover $(V_{x_n}x_n)_{n=1}^N.$ Let $V_e$ be defined by the formula
\begin{gather}
V_e := \bigcap_{n=1}^N\ V_{x_n}.
\label{definitionofVeSudakov}
\end{gather}

For every $x\in K$ let $ind(x)\in \{1,\ldots,N\}$ denote an index such that $x\in V_{x_{ind(x)}}x_{ind(x)}.$ Consequently, the fact that 
$$\forall_{x\in K}\ x\in V_{x_{ind(x)}}x_{ind(x)} \subset U_{x_{ind(x)}}x_{ind(x)}$$

\noindent
coupled with (\ref{equicontSudakov}) implies that
\begin{gather}
\forall_{\substack{x\in K\\ f\in\Ffamily}}\ |f(x) - f(x_{ind(x)})| < \eps.
\label{ineq1Sudakov}
\end{gather} 

\noindent
Furthermore, since
$$\forall_{\substack{y\in V_e\\ x\in K}}\ yx \in V_eV_{x_{ind(x)}}x_{ind(x)} \stackrel{(\ref{definitionofVeSudakov})}{\subset} V_{x_{ind(x)}}^2x_{ind(x)}\subset U_{x_{ind(x)}}x_{ind(x)},$$

\noindent
then (again using (\ref{equicontSudakov})) we have
\begin{gather}
\forall_{\substack{y\in V_e\\ x\in K\\ f\in\Ffamily}}\ |f(yx) - f(x_{ind(x)})| < \eps.
\label{ineq2Sudakov}
\end{gather}

\noindent
Finally, we have
\begin{gather*}
\forall_{\substack{y\in V_e\\ f\in\Ffamily}}\ \sup_{x\in K}\ |f(yx) - f(x)| \leqslant \sup_{x\in K}\ |f(yx) - f(x_{ind(x)})| + \sup_{x\in K}\ |f(x_{ind(x)}) - f(x)| \stackrel{(\ref{ineq1Sudakov}),\ (\ref{ineq2Sudakov})}{\leqslant} 2\eps,
\end{gather*}

\noindent
which concludes the proof.
\end{proof}

Our next theorem demonstrates that under certain circumstances, the assumption (\textbf{AA1}) in Theorem \ref{AAforC0X} becomes redundant.

\begin{thm}
Let $X$ be a locally compact (not necessarily abelian) group such that for any open neighbourhood $U_e$ of the neutral element $e\in X$ there exists an element $x_*\in U_e$ such that the sequence $(x_*^n)_{n\in\naturals}$ is not contained in any compact set. If a family $\Ffamily\subset C_0(X)$ is equicontinuous at every point and equivanishing then it is pointwise bounded, i.e. (\textbf{AA2}) and (\textbf{AA3}) together imply (\textbf{AA1}). 
\label{AASudakov}
\end{thm}
\begin{proof}
Let $\Ffamily\subset C_0(X)$ be a family, which is both equicontinuous at every point and equivanishing. By Lemma \ref{technicallemmasudakov} there exists an open neighbourhood $V_e$ of the neutral element such that 
\begin{gather}
\forall_{\substack{y\in V_e\\ f\in\Ffamily}}\ \sup_{x\in K}\ |f(yx) - f(x)| < 1.
\label{choiceofUe}
\end{gather}

\noindent
Furthermore there exists $K\Subset X$ such that 
\begin{gather}
\sup_{\substack{x\in X\backslash K\\ f\in\Ffamily}}\ |f(x)| < 1.
\label{equivanish1}
\end{gather}

For the set $V_e$ we pick $x_*\in V_e$ such that the sequence $(x_*^n)_{n\in\naturals}$ is not contained in any compact set. Since we have
\begin{gather*}
\forall_{f\in\Ffamily}\ \sup_{x\in K}\ |f(x)| \leqslant \sup_{x\in K}\ |f(x_*x) - f(x)| + \sup_{x\in K}\ |f(x_*x)| \stackrel{(\ref{choiceofUe})}{<} 1 + \sup_{x\in K}\ |f(x_*x)| = 1 + \sup_{x\in x_*K}\ |f(x)|,
\end{gather*}

\noindent
then by an inductive reasoning we obtain that
\begin{gather}
\forall_{\substack{f\in\Ffamily\\ n\in\naturals}}\ \sup_{x\in K}\ |f(x)| \leqslant n + \sup_{x\in x_*^nK}\ |f(x)|.
\label{byinductivereasoning}
\end{gather}

\noindent
It remains to note that there exists $N\in\naturals$ such that $x_*^NK\cap K = \emptyset$ since otherwise this would lead to $(x_n)\subset KK^{-1},$ which contradicts our assumption. Finally, we have
$$\forall_{f\in\Ffamily}\ \|f\|_{\infty} \leqslant \sup_{x\in K}\ |f(x)| + \sup_{x\in X\backslash K}\ |f(x)|\stackrel{(\ref{byinductivereasoning}), (\ref{equivanish1})}{<} N + \sup_{x\in x_*^NK}\ |f(x)| + 1 \leqslant N+2,$$

\noindent
which concludes the proof. 
\end{proof}

\begin{cor}
A family $\Ffamily\subset C_0(\reals^N)$ is relatively compact (with respect to $\|\cdot\|_{\infty}-$norm) if and only if 
\begin{itemize}
	\item $\Ffamily$ is equicontinuous at every point,
	\item $\Ffamily$ is equivanishing. 
\end{itemize}
\end{cor}

We have finally reached the climax of the paper $-$ the characterization of relatively compact families in $L^1(G)$ via the Fourier transform. The theorem below serves as $L^1(G)-$counterpart of the classic Pego theorem demonstrated in \cite{Pego} and generalized in \cite{Gorka} and \cite{GorkaKostrzewa}. 

\begin{thm}(Pego theorem for $L^1(G)$)\\
A family $\Ffamily\subset L^1(G)$ is relatively compact if and only if 
\begin{description}
	\item[\hspace{0.4cm} (P1)] $\widehat{\Ffamily}$ is bounded in $C_0(\widehat{G}),$
	\item[\hspace{0.4cm} (P2)] $\widehat{\Ffamily}$ is equicontinuous at every point (character), i.e. for every $\chi\in\widehat{G}$ and $\eps>0$ there exists an open neighbourhood $U_{\chi}\in\widehat{G}$ of $\chi$ such that
	$$\forall_{\substack{\eta\in U_{\chi}\\ f\in\Ffamily}}\ |\widehat{f}(\eta) - \widehat{f}(\chi)| < \eps,$$
	\item[\hspace{0.4cm} (P3)] $\widehat{\Ffamily}$ is equivanishing, i.e. for every $\eps>0$ there exists $K\Subset \widehat{G}$ such that 
	$$\sup_{\chi\in\widehat{G}\backslash K}\ |\widehat{f}(\chi)| < \eps.$$
\end{description}

Furthermore, if $\widehat{G}$ is such that for any open neighbourhood $U_{\mathds{1}}$of the identity $\mathds{1}$ there exists an element $\chi_*\in U_{\mathds{1}}$ such that the sequence $(\chi_*^n)_{n\in\naturals}$ is not contained in any compact set, then the condition (\textbf{P1}) is redundant.
\label{Pegotheorem}
\end{thm}
\begin{proof}
Let $f\in\Ffamily$. Then 
\begin{gather*}
\Omega^{-1}\circ\Psi(f) \stackrel{(\ref{definitionofOmega})}{=} (\Gamma^{-1}\circ\Theta)^{-1}(\Psi_f) = \Theta^{-1}\circ\Gamma(\Psi_f) = \Theta^{-1}(\widehat{\Psi_f}) \stackrel{(\ref{inverseTheta})}{=} \widehat{\Psi_f}\circ \Psi_{\Delta}^{-1}\circ M.
\end{gather*}

\noindent
Furthermore, we have
\begin{gather*}
\forall_{\chi\in\widehat{G}}\ \widehat{\Psi_f}\circ\Psi_{\Delta}^{-1}\circ M(\chi) \stackrel{(\ref{definitionM})}{=} \widehat{\Psi_f}\circ\Psi_{\Delta}^{-1}(m_{\chi}) \stackrel{(\ref{inversePsi})}{=} \widehat{\Psi_f}(m_{\chi}\circ\Psi^{-1})\\
= m_{\chi}\circ\Psi^{-1}(\Psi_f) = m_{\chi}(f) = \widehat{f}(\chi),
\end{gather*}

\noindent
which leads to the conclusion that $\Omega^{-1}\circ\Psi(\Ffamily) = \widehat{\Ffamily}.$ By Theorem \ref{AAforC0X} the family $\Omega^{-1}\circ \Psi(\Ffamily)$ (equivalently $\Ffamily$) is relatively compact in $C_0(\widehat{G})$ (respectively in $L^1(G)$) if and only if \textbf{(P1)}, \textbf{(P2)} and \textbf{(P3)} are satisfied. This demonstrates the first part of the theorem. The second part easily follows from Theorem \ref{AASudakov}.
\end{proof}

\begin{cor}
A family $\Ffamily\subset L^1(\reals^N)$ is relatively compact if and only if 
\begin{itemize}
	\item $\widehat{\Ffamily}$ is equicontinuous,
	\item $\widehat{\Ffamily}$ is equivanishing.
\end{itemize}
\end{cor}

\section*{Acknowledgements}

\end{document}